\documentclass{amsart}
\usepackage{amsthm}
\usepackage[english]{babel}
\usepackage[autostyle]{csquotes}
\newtheorem{theorem}{Theorem}[section]
\newtheorem{lemma}[theorem]{Lemma}
\newtheorem{notation}[theorem]{Notation}
\newtheorem{proposition}[theorem]{Proposition}

\newtheorem{conjecture}[theorem]{Conjecture}
\theoremstyle{definition}
\newtheorem{definition}[theorem]{Definition}
\newtheorem{example}[theorem]{Example}

\newtheorem{innercustomthm}{Theorem}
\newenvironment{customthm}[1]
{\renewcommand\theinnercustomthm{#1}\innercustomthm}
{\endinnercustomthm}

\theoremstyle{remark}
\newtheorem{remark}[theorem]{Remark}

\numberwithin{equation}{section}
\usepackage{amssymb}
\usepackage{amsmath}
\begin{document}
	
	\title[Quadratic Polynomials]{Irreducibility of iterates of post-critically finite quadratic polynomials over $\mathbb{Q}$}
	%    Remove any unused author tags.
	
	%    author one information
	\author[V.Goksel]{Vefa Goksel}
	\address{Mathematics Department\\ University of Wisconsin\\
		Madison\\
		WI 53706, USA}
	\email{goksel@math.wisc.edu}

	\subjclass[2000]{Primary }
	%    For articles to be published after 1 January 2010, you may use
	%    the following version:
	%\subjclass[2010]{Primary }
	
	\keywords{iteration, quadratic polynomial, post-critically finite}
	
	\date{}
	
	\dedicatory{}
	
	\begin{abstract}
	In this paper, we classify, up to three possible exceptions, all monic, post-critically finite quadratic polynomials $f(x)\in \mathbb{Z}[x]$ with an iterate reducible module every prime, but all of whose iterates are irreducible over $\mathbb{Q}$. In particular, we obtain infinitely many new examples of the phenomenon studied in \cite{Jones}. While doing this, we also find, up to three possible exceptions, all integers $a$ such that all iterates of the quadratic polynomial $(x+a)^2-a-1$ are irreducible over $\mathbb{Q}$, which answers a question posed in \cite{AyadMcdonald}, except for three values of $a$. Finally, we make a conjecture that suggests a necessary and sufficient condition for the stability of any monic, post-critically finite quadratic polynomial over any field of characteristic $\neq 2$.	
	\end{abstract}
	\subjclass[2010]{Primary 11R09, 37P15}
	
	\keywords{iteration, quadratic polynomial, post-critically finite}
	
	\date{}
	\maketitle
    \section{Introduction}
    
    Hilbert gave examples of irreducible polynomials $f(x)\in \mathbb{Z}[x]$ which are reducible modulo every prime, namely any irreducible polynomial of the form $x^4+2ax^2+b^2$, where $a,b\in \mathbb{Z}$. Moreover, polynomials of the form $x^{2^n}+1$ for $n\geq 2$ share the same property as well. In \cite{Jones}, Jones gives a generalization of this, and constructs infinitely many infinite families of such examples. He achieves this by giving criteria ensuring that a quadratic polynomial $f \in \mathbb{Z}[x]$ has all iterates irreducible over Q, but a given iterate reducible modulo every prime. The strategy of \cite{Jones} is to consider carefully selected conjugates of polynomials of the form $x^2 + c$. In this paper, we classify, up to three possible exceptions, all monic, post-critically finite (see below for a definition) quadratic polynomials $f(x)\in \mathbb{Z}[x]$ with an iterate reducible module every prime, but all of whose iterates are irreducible over $\mathbb{Q}$. Note that having an iterate reducible modulo every prime implies that all subsequent iterates also have the same property. We say that the quadratic polynomial $f$ is \textbf{post-critically finite} or \textbf{PCF} for short, if the orbit of its critical point under the iteration of $f$ is finite. It follows from a straightforward calculation that all the monic PCF quadratic polynomials with integer coefficients are conjugate to $x^2$, $x^2-1$ or $x^2-2$ under the map $x\rightarrow x+a$, $a\in \mathbb{Z}$. So, our method proceeds by classifying all conjugates over Z of the maps $x^2$, $x^2 - 1$, and $x^2 - 2$ that have the desired properties. In particular, this leads to infinitely many new examples of the above phenomenon not covered by the criteria given in \cite{Jones}. \\

    Precisely, we prove the following theorem:
    
    \begin{theorem}\label{thm:neat}
    	Let $f(x)\in \mathbb{Z}[x]$ be a monic PCF quadratic polynomial. Let $S = \{9, 9801, 332929\}$, and suppose that $f(x) \neq (x-m^2)^2 + m^2-1$ for any $m\in S$. Then $f^n(x)\in \mathbb{Z}[x]$ is irreducible for all $n$, and there exists $N\in \mathbb{N}$ such that $f^n(x)$ is reducible modulo every prime for all $n\geq N$ if and only if $f(x)$ has one of the following forms:\\
    	
    	\item[1)] $f(x)=(x-b^2)^2+b^2$ for some $b\in \mathbb{Z}$ such that $b\neq 2k^2$ for any $k\in \mathbb{Z}$.\\
    	\item[2)] $f(x)=(x+b^2)^2-b^2-1$ for some $b\in \mathbb{Z}$.\\
    	\item[3)] $f(x)=(x-b^2)^2+b^2-1$ for some $b\in \mathbb{Z}$ such that $b\neq 2k^2-1$ for any $k\in \mathbb{Z}$.\\
    	\item[4)] $f(x)=(x-b^2-1)^2+b^2$ for some $b\in \mathbb{Z}$ such that $b\neq 2(r^2\pm r\sqrt{2(r^2-1)})^2$ for any integer solution $r$ of the Pell equation $2r^2-t^2=2$.\\
    	\item[5)] $f(x)=(x+2-b^2)^2+b^2-4$ for some $b\in \mathbb{Z}$ such that $b\neq 2v^2-2$ for any $v\in \mathbb{Z}$.
    	
    \end{theorem}
    
    Note that the exceptions such as $b\neq 2k^2$, $b\neq 2k^2-1$, $b\neq 2(r^2\pm r\sqrt{2(r^2-1)})^2$ and $b\neq 2v^2-2$ come from the fact that those values make one of the first few iterates of the corresponding polynomials reducible over $\mathbb{Q}$. We say that a quadratic polynomial $f(x)\in \mathbb{Z}[x]$ is \textbf{stable} if all of its iterates are irreducible over $\mathbb{Q}$. In particular, our work implies that stability is a rigid property for the monic PCF quadratic polynomials with integer coefficients. In other words, for these polynomials, irreducibility of first few iterates implies the irreducibility of all the iterates. We will now state this phenomenon more concretely:\\
    
    Set $f_a(x) = (x+a)^2-a$, $g_a(x) = (x+a)^2-a-1$, $h_a(x) = (x+a)^2-a-2$, where $a\in \mathbb{Z}$. We prove the following theorem which shows that the stability is a rigid property for these families of polynomials:
    
    \begin{theorem}\label{thm:neat3}
    Let $S = \{9,9801,332929\}$, and suppose that $g_a(x) \neq (x-m^2)^2 + m^2-1$ for any $m\in S$. Then we have the following:\\
    \item[(i)] All the iterates of $f_a$ are irreducible $\iff$ $f_a^2$ is irreducible.\\
    
    \item[(ii)] All the iterates of $g_a$ are irreducible $\iff$ $g_a^3$ is irreducible.\\
    
     \item[(iii)] All the iterates of $h_a$ are irreducible $\iff$ $h_a^2$ is irreducible.
    \end{theorem}
    
    Note that the part (ii) of Theorem $1.2$ also answers a question posed by Ayad and McQuillan in \cite{AyadMcdonald}, except for the excluded three values of $a$. See Remark $3.11$ for more details about this.\\
    
    The structure of the paper is as follows: In Section 2, we give some preliminary definitions and results. In Section 3, we prove the main results. In Section 4, we first generalize one of the results that appears in Section 3, and then make a conjecture regarding the stability of monic, PCF quadratic polynomials over any field of characteristic $\neq 2$.
	
	\section{Preliminaries}
	Let $K$ be a field, $f(x)\in K[x]$ a quadratic polynomial. For any $n\geq 1$, we denote by $f^n(x)$ the $n$th iterate of $f(x)$. We also make the convention that $f^0(x) = x$. Let $c$ be the critical point of $f$. Then the \textbf{post-critical orbit} of $f$ is given by
	$$O_f = \{f(c),f^2(c),\dots\}.$$  
	When this set is finite, we say $f$ is \textbf{post-critically finite}, or \textbf{PCF} for short. In this case, $|O_f|$ is the size of the post-critical orbit, which we denote by $o_f$. We also define the \textbf{tail} of $f$, $T_f$, by the set
	$$T_f = \{f^i(c) \text{ }|\text{ } i\geq 1, f^i(c)\neq f^k(c) \text{ for any } i\neq k\}.$$
	Hence, $|T_f|$ is the tail size of $f$, which we denote by $t_f$.
	\begin{example}
	Take $K=\mathbb{Q}$, and $f(x) = x^2-2\in \mathbb{Q}$. The critical point is $c=0$, hence the post-critical orbit becomes $O_f = \{-2,2\}$. In this case, we have $o_f = 2$, $t_f = 1$.
	\end{example}

    Since we will study the iterates of quadratic polynomials modulo primes, next we need to give some definitions related to the polynomials over finite fields.
    Throughout, we denote by $\mathbb{F}_q$ the finite field of size $q$, where $q$ is a prime power.
    \begin{definition}
    Let $f(x)\in \mathbb{F}_q[x]$ be a quadratic polynomial with post-critical orbit $O_f$, and $g(x)\in \mathbb{F}_q[x]$ be an irreducible polynomial. We define the type of $g(x)$ at $\beta$ to be $s$ if $g(\beta)$ is a square in $\mathbb{F}_q$, and $n$ if it is not a square. The type of $g$ is a string of length $|O_f|$ whose $k$th entry is the type of $g(x)$ at the $k$th entry of $O_f$. 
    \end{definition}
\begin{example}
Take $g(x)= (x+1)^2-2\in \mathbb{F}_5[x]$, and $f(x) = x^2-2\in \mathbb{F}_5[x]$. We have $O_f = \{-2,2\}$. Since $g(-2) = -1$ is a square in $\mathbb{F}_5$, and $g(2) = 2$ is not a square in $\mathbb{F}_5$, we conclude that $g$ has type $sn$.
\end{example}

We now quote the following lemma from \cite{JonesBoston}, which will be one of the building stones of our paper:

\begin{lemma}\cite{JonesBoston}
Let $K$ be a field of odd characteristic, $f(x)\in K[x]$ a monic, quadratic polynomial with the critical point $c\in K$. Suppose that $g\in K[x]$ is another monic polynomial such that $g\circ f^{n-1}(x)$ has degree $d$, and is irreducible over $K$ for some $n\geq 1$. Then $g\circ f^n$ is irreducible over $K$ if $g(f^n(c))$ is not a square in $K$. If $K$ is finite, we can replace \enquote{if} with \enquote{if and only if}. 
\end{lemma}
The following remark is a straightforward consequence of Lemma $2.4$:
\begin{remark}
Let $q$ be an odd prime power, and $f(x)\in \mathbb{F}_q[x]$ an irreducible quadratic polynomial with post-critical orbit $O_f$. Then, by Lemma $2.4$, all the iterates of $f$ are irreducible over $\mathbb{F}_q$ if and only if $f$ has type $nn\dots n$. Moreover, suppose $f$ does not have type $nn\dots n$. For some $i\geq 1$, let $s$ appear in the $i$th entry of the type of $f$ for the first time. Then $f^{i+1}$ is the first iterate of $f$ that is reducible.
\end{remark}
Throughout the paper, we will use the framework of \cite{AyadMcdonald}. To do this, we need to recall some technical notions they use. We will give all of them in the next definition:\\

\begin{definition}
\cite{AyadMcdonald} Let $K$ be a field, and $f(x) = x^2+ax+b$. We denote the discriminant of $f$ by $d_f = a^2 -4b$. We also define another invariant $\delta_f$ of $f$ by $\delta_f = -d_f +2a$.\\

Set $\delta_0 = \frac{\delta_f}{4}$. We define the sequence $\{d_i\}_{i\geq 0}$ by $d_0 = \frac{d_f}{4}$, and the recursion relation $d_i = -\delta_0+\sqrt{d_{i-1}}$ for $i\geq 1$. Using this sequence, we define a sequence $\{K_i\}_{i\geq 0}$ of fields by setting $K_0 = K$, and $K_i = K(d_i)$ for $i\geq 1$.\\

We would like to include the following translation of the notation in the previous paragraph, for the readers more familiar with the framework of \cite{Jones}: By completing the square, one can write $f$ of the form $f(x) = (x-\gamma)^2+\gamma + m$. Then $\delta_0$ is $m$, $d_0$ is $-f(\gamma)$ ($=-\gamma-m$), and $\sqrt{d_i}$ is a root of $f^{i+1}(x+\gamma)$.\\  

Finally, let $g_0(x) = -x$, $g_1(x) = x^2+\delta_0$ and $g_{r+1}(x) = g_1(g_r(x))$, so $g_{r+1}(x) = [g_r(x)]^2+\delta_0$ for $r\geq 1$. We define the sequence $\{g_r\}_{r\geq 0}\subseteq K$ by $g_r = g_r(\delta_0)$ for all $r\geq 0$. In other words, we have $g_0 = -\delta_0$, $g_1 = \delta_0^2 + \delta_0$, $g_{r+1} = g_r^2 + \delta_0$ for $r\geq 1$.\\
\end{definition}

\begin{example}
Let $f(x)=x^2-x-1$. If we calculate the quantities defined above for $f(x)$, we get $d_f=5$, $\delta_f=-7$, $\delta_0 = \frac{-7}{4}$. Hence the sequence $\{d_i\}_{i\geq 0}$ becomes $d_0 = \frac{5}{4}$, $d_i = \frac{7}{4}+\sqrt{d_{i-1}}$ for $i\geq 1$. Finally, the sequence $\{g_r\}_{r\geq 0}$ becomes $g_0=\frac{7}{4}$, $g_1 = \frac{21}{16}$, $g_{r+1} = g_r^2-\frac{7}{4}$ for $r\geq 1$.
\end{example}

Having given Definition $2.6$, we now quote a theorem from \cite{AyadMcdonald}, which will be an important ingredient throughout the paper:

\begin{theorem}\cite{AyadMcdonald}
Let $n\geq 1$ and let $f^n(x)$ be irreducible in $K[x]$. If $f^{n+1}(x)$ is reducible over $K$, then for every $r$, $0\leq r \leq n-1$, there exist elements $a_r, b_r\in K_{n-r-1}$ such that $g_r^2-b_r^2 = d_{n-r-1}$ and $a_r^2 = \frac{g_r + b_r}{2}$. Furthermore, for every $r$ such that $1\leq r \leq n-1$, we have \begin{equation}
b_{r-1}=\pm (a_r-\frac{\sqrt{d_{n-r-1}}}{2a_r}).
\end{equation}
Conversely, if there exist elements $a_r$ and $b_r$ with these properties, then $f^{n+1}(x)$ is reducible in $K[x]$.
\end{theorem}
\begin{notation}
Let $a\in \mathbb{Z}$. We define $X_a$ to be a set of polynomials, given by
$$X_a = \{f_a,g_a,h_a\},$$ 
where $f_a, g_a, h_a\in \mathbb{Z}[x]$ are as in Theorem $1.2$.
\end{notation}
\begin{remark}
Note that by the discussion preceding Theorem $1.1$, the union $\cup_{a\in \mathbb{Z}} X_a$ is equal to the set of all monic PCF quadratic polynomials with integer coefficients.	
\end{remark}
	\section{Proof of the main results}
	We start with the following lemma which gives a simple characterization of the PCF quadratic polynomials with integer coefficients and an iterate reducible modulo every prime. From now on, for any polynomial $f(x)\in \mathbb{Z}[x]$ and prime $p$, $\bar{f}(x)\in \mathbb{F}_p[x]$ denotes the polynomial $f$ reduced mod $p$.
	\begin{lemma}
		Let $f(x)\in X_a$ for some $a\in \mathbb{Z}$. Suppose that $f(x)\in \mathbb{Z}[x]$ is irreducible. Then we have the following: There exists $N\in \mathbb{N}$ such that $f^n(x)$ is reducible modulo all primes for all $n\geq N$ if and only if there does not exist any odd prime $p$ such that $\bar{f}(x)\in \mathbb{F}_p[x]$ has type $nn\dots n$.
	\end{lemma}
	
	\begin{proof}
    Recall that for some $a\in \mathbb{Z}$, $f$ is equal to $f_a$, $g_a$ or $h_a$, where $f_a, g_a, h_a$ are as in Theorem $1.2$. It is clear that for any $a\in \mathbb{Z}$, all three of these polynomials are already reducible modulo $2$. Hence, it suffices to consider $f$ modulo odd primes. Note that ($\implies$) part is clear by Remark $2.5$. We now prove the other direction. Using Remark $2.5$, for any odd prime $p$, if $\bar{f}(x)\in \mathbb{F}_p[x]$ does not have type $nn\dots n$, then there exists $N_p\in \mathbb{Z}$ such that $\bar{f^i}(x)\in \mathbb{F}_p[x]$ is reducible for all $i\geq N_p$. Since $f$ is PCF with orbit size $|O_f|$, by the second part of Remark $2.5$, we have $N_p\leq |O_f|+1$ for each odd prime $p$. Now, taking $N = |O_f| + 1$, the result directly follows.
	\end{proof}

	To be able to use Lemma 3.1, we will first determine all $f(x)\in \mathbb{Z}[x]$ such that $f\in X_a$ for some $a$ and there does not exist any odd prime $p$ such that $\bar{f}(x)\in \mathbb{F}_p[x]$ has type of the form $n\dots n$. For simplicity, we give the following definition:
	
	\begin{definition}
		Let $f(x)\in X_a$ for some $a\in \mathbb{Z}$. Suppose $f(x)\in \mathbb{Z}[x]$ is irreducible. We say $f$ is a \textbf{special type polynomial} if there does not exist any odd prime $p$, for which $\bar{f}(x)\in \mathbb{F}_p[x]$ is irreducible and has type $n\dots n$.
	\end{definition}
	
	The following lemma gives all special type polynomials:
	
	\begin{lemma}
		Let $a\in \mathbb{Z}$, and $f(x)\in X_a$. Then $f$ is a special type polynomial if and only if it  has one of the forms $(x-b^2)^2+b^2$, $(x+b^2)^2-b^2-1$,  $(x-b^2-1)^2+b^2$, $(x-b^2)^2+b^2-1$ or $(x+2-b^2)^2+b^2-4$, where $b\in \mathbb{Z}$.
		
	\end{lemma}
	
	\begin{proof}
		We will look at each form of polynomials in $X_a$ separately:
		
		\item[i)] $f(x)=(x+a)^2-a$ : The critical point is $-a$, and it follows that the post-critical orbit becomes $\{-a\}$. Hence, type is determined by the set $\{f(-a)\}=\{-a\}$. For f to be special type, $-a$ must be square in $\mathbb{F}_p$ for every odd prime $p$ such that $\bar{f}(x)\in \mathbb{F}_p[x]$ is irreducible, which holds if and only if $a=-b^2$ for some $b\in \mathbb{Z}$.\\
		\item[ii)] $f(x)=(x+a)^2-a-1$ : The critical point is $-a$, and the post-critical orbit becomes $\{-a-1, -a\}$. So, type is determined by the set $\{f(-a-1),f(-a)\}=\{-a,-a-1\}$. For $f$ to be special type, for any given odd prime $p$ such that $\bar{f}(x)\in \mathbb{F}_p[x]$ is irreducible (i.e., $(\frac{a+1}{p}) = -1$), either $-a$ or $-a-1$ need to be square in $\mathbb{F}_p$. It easily follows that all such values of $a$ are $a=b^2$, $a=-b^2$ or $a=-b^2-1$ for some $b\in \mathbb{Z}$.\\
		\item[iii)] $f(x)=(x+a)^2-a-2$: The post-critical orbit is $\{-a-2,-a+2\}$. Similarly, type is determined by the set $\{f(-a-2),f(-a+2)\}=\{2-a\}$. For $f$ to be special type, $2-a$ must be a square in $\mathbb{F}_p$ for any odd prime $p$ such that $\bar{f}(x)\in \mathbb{F}_p[x]$ is irreducible. It easily follows that all such $a$ values are $a=2-b^2$ for some $b\in \mathbb{Z}$, which completes the proof.
	\end{proof}

   Next, we recall some terminology from the theory of polynomial iteration:
   
   \begin{definition}
   	Let $K$ be a field, and $f(x)\in K[x]$ a polynomial. We say that $f(x)$ is \textbf{stable} if $f^n(x)$ is irreducible over $K$ for all $n\geq 1$.
   \end{definition} 
	
	We now determine, up to three possible exceptions, all monic, stable PCF quadratic polynomials $f(x)\in \mathbb{Z}[x]$. We will then combine this with Lemma 3.1 and Lemma 3.3 to prove Theorem $1.1$.\\
	
	To do this, we will give three separate results for the three forms $f_a,g_a,h_a$.
	
	\begin{proposition}
		Let $f(x)=f_a(x)\in \mathbb{Z}[x]$ be irreducible for some $a\in \mathbb{Z}$. Then $f$ is stable if and only if $a\neq -4u^4$ for any $u\in \mathbb{Z}$.
	\end{proposition}
	
	\begin{proof}
		Recall that $f_a(x) = (x+a)^2-a$. Note that the $n$th iterate of $f_a$ for any $n$ is $(x+a)^{2^n}-a$. It is a well-known fact in field theory that for a field $F$, for any $c\in F$, and $k\geq 1$ an integer, $x^k-c\in F[x]$ is irreducible if and only if $c\notin F^p$ for all primes $p|k$ and $c\notin -4F^4$ when $4|k$ (see Thm 8.1.6 in \cite{Karp}). Taking $F=\mathbb{Q}$, making the change of variable $y=x+a$, and noting that $4|2^n$ in our case, the result directly follows.
	\end{proof}
	
	\begin{remark}
		It is clear from the proof that $f_a$ is stable if and only if $f_a^2$ is irreducible.
	\end{remark}
	\begin{proposition}
		Let $f(x)=h_a(x)\in \mathbb{Z}[x]$ be irreducible for some $a\in \mathbb{Z}$. Then $f$ is stable if and only if $a\neq 2-(2v^2-2)^2$ for any $v \in \mathbb{Z}$.
	\end{proposition}
	
	\begin{proof}
		Recall that $h_a(x) = (x+a)^2-a-2$. Using the notation in Definition $2.6$, we have $d_f=4a+8$, $\delta_f=-8$. By (\cite{AyadMcdonald}, Theorem $3$), $f$ is stable if $16 \not|\text{ } d_f$. So, assume $16| d_f$. Then by (\cite{AyadMcdonald}, Remark $4$), since $\delta_f = -8$, we have $f$ is stable if and only if $f^2$ is irreducible. So, it is enough to understand the values of $a$ that make $f^2$ reducible. Noting that $f$ is irreducible, if $f^2$ is reducible, using (\cite{JonesBoston}, Proposition $2.6$), 
		$$f^2(x)=h(x+a)h(-(x+a))$$
		 for some monic quadratic polynomial $h(x)\in \mathbb{Z}[x]$. Putting $h(x)=x^2+vx+w$, and equating the coefficients on both sides, we get \begin{equation}
		2w-v^2=-4,
		\end{equation}and
		 \begin{equation}
		 -a+2=w^2.
		 \end{equation} Note that $v$ must be even. If we replace $v$ by $2v$, $(3.1)$ becomes \begin{equation}
		 w-2v^2=-2.
		 \end{equation} Combining $(3.2)$ with $(3.3)$, we get $a= 2 - (2v^2-2)^2$. It is clear that $f^2$ will be reducible if $a$ is of this form. Thus, $f^2(x)$ is irreducible if and only if $a\neq 2-(2v^2-2)^2$ for any $v \in \mathbb{Z}$, which completes the proof.
	\end{proof}
	
	\begin{theorem}\label{thm:neat2}
		Let $f(x)=g_a(x)\in \mathbb{Z}[x]$ be irreducible for some $a\in \mathbb{Z}$. Let $S = \{9, 9801, 332929\}$, and suppose that $f(x) \neq (x-m^2)^2 + m^2-1$ for any $m\in S$. Then $f$ is stable if and only if $a\neq -4(r^2\pm r\sqrt{2(r^2-1)})^4-1$ for any integer solution $r$ of the Pell equation $2r^2-t^2=2$ and $a\neq -(2k^2-1)^2$ for any $k\in \mathbb{Z}$.
	\end{theorem}  
	To prove this theorem, we will first prove two different lemmas, which together will directly imply Theorem 3.8.\\
	\begin{lemma}
		Let $f(x)=g_a(x)\in \mathbb{Z}[x]$ be irreducible for some $a\in \mathbb{Z}$. Then $f^3$ is irreducible if and only if $a\neq -4(r^2\pm r\sqrt{2(r^2-1)})^4-1$ for any integer solution $r$ of the Pell equation $2r^2-t^2=2$ and $a\neq -(2k^2-1)^2$ for any $k\in \mathbb{Z}$.
	\end{lemma}
	\begin{proof}
		Recall that $g_a(x) = (x+a)^2-a-1$. We will first determine the values $a$ that make $f^2$ irreducible. For this, suppose that $f^2$ is reducible. Noting that $f$ is irreducible and using (\cite{JonesBoston}, Proposition $2.6$), we have
	     $$f^2(x)=h(x+a)h(-(x-a))$$for some monic quadratic polynomial $h(x)\in \mathbb{Z}[x]$. Setting $h(x)=x^2+kx+l$, and equating the coefficients in both sides, get 
		$$a=-l^2$$
		and 
		$$2l-k^2=-2.$$
		Note that $k$ is even, so replace $k$ by $2k$, and obtain
		
		$$l-2k^2=-1,$$
		which gives
		$$a = -(2k^2-1)^2.$$
		So, if $a\neq -(2k^2-1)^2$ for any $k\in \mathbb{Z}$, $f^2$ is irreducible. It is also clear that if $a$ is of this form, then $f^2$ is reducible. Hence, we get that $f^2$ is irreducible if and only if $a\neq -(2k^2-1)^2$ for any $k\in \mathbb{Z}$.\\
		
		We now assume that $f^2$ is irreducible and $f^3$ is reducible. Again using (\cite{JonesBoston}, Proposition $2.6$), we have 
		$$f^3(x)=h(x+a)h(-(x+a))$$
		 for some monic quartic $h(x)\in \mathbb{Z}[x]$. Setting $h(x)=x^4+kx^3+lx^2+mx+n$, and equating the coefficients in both sides, get \begin{equation}
		2l-k^2=-4,
		\end{equation} 
		\begin{equation}
		l^2+2n-2km=4,
		\end{equation}
		\begin{equation}
	    2ln-m^2=0,
		\end{equation} and
		\begin{equation}
		-(n^2+1)=a.
		\end{equation} If we consider $(3.4)$, for $l$ to be an integer, $k$ must be even, so put $k=2r$. Using this in $(3.4)$ and $(3.6)$, get \begin{equation}
		l=2r^2-2
		\end{equation} and 
		\begin{equation}
		n=\frac{m^2}{4r^2-4}.
		\end{equation} Using $(3.8)$ and $(3.9)$ in $(3.5)$, and simplifying, get 
		\begin{equation}
		8m^2-[16r(4r^2-4)]m+(4r^2-4)^3-16(4r^2-4)=0.
		\end{equation} Considering this as a quadratic equation in $m$, to have an integer solution, the discriminant needs to be a square, i.e. (after simplification) $(4r^2-4)(512r^4)$ needs to be square, i.e. $2(r^2-1)$ needs to be a square. Therefore, there exists an integer $t$ such that 
		$$2r^2-t^2=2,$$
		which is a Pell equation, solutions of which could easily be given in terms of a fundamental unit in $\mathbb{Z}[\sqrt{2}]$. Writing $m$ in terms of $r$ using $(3.10)$, and using $(3.7)$ and $(3.9)$, after simplifying, get \begin{equation}
		a = -4(r^2\pm r\sqrt{2(r^2-1)})^4-1.
		\end{equation} Hence, if $f^2$ is irreducible and $a \neq -4(r^2\pm r\sqrt{2(r^2-1)})^4-1$ for any $r\in \mathbb{Z}$, then $f^3$ irreducible. It is also clear that $f^3$ is reducible if $a$ is of this form. Hence, assuming that $f^2$ is irreducible, we get that $f^3$ is irreducible if and only if $a \neq -4(r^2\pm r\sqrt{2(r^2-1)})^4-1$ for any $r\in \mathbb{Z}$. Combining this with the first part, the proof of the lemma is complete.
	\end{proof}
	\begin{lemma}
		Let $f(x)=g_a(x)\in \mathbb{Z}[x]$ be irreducible for some $a\in \mathbb{Z}$. Let $S = \{9, 9801, 332929\}$, and suppose that $f(x) \neq (x-m^2)^2 + m^2-1$ for any $m\in S$. Then $f$ is stable if and only if $f^3$ is irreducible.
	\end{lemma}
	\begin{proof}
		One direction is obvious, so assume $f^3$ is irreducible, and we will show that $f$ is stable. Assume for the sake of contradiction that there exists $N\geq 3$ such that $f^N$ is irreducible, but $f^{N+1}$ is reducible. We will use the notation in Definition $2.6$ throughout the proof. First note that by direct calculation, we have $d_f=4a+4$, $\delta_f = -4$. By (\cite{AyadMcdonald}, Theorem $3$), $f$ is stable if $d$ is not divisible by $16$. So, assume $16|d$. With the notation in Definition $2.6$, we have $\delta_0 = \frac{\delta}{4}=-1$, hence $g_0=1, g_1=0, g_2=-1, g_3=0, \dots$ In other words, we have \begin{equation}
		g_0=1, \text{ }g_{2k-1}=0,\text{ and } g_{2k}=-1 \text{ for all } k\geq 1.
		\end{equation} 
		
		First assume that $N\geq 3$ above is odd. So, $g_{N-1}=-1$. Since $f^N$ is irreducible and $f^{N+1}$ is reducible, taking $n=N$ and $r=N-1$ in Theorem $2.8$, there exist integers $a_{N-1}, b_{N-1}$ such that 
		$$g_{N-1}^2-b_{N-1}^2 =1-b_{N-1}^2=\frac{d_f}{4}=a+1$$
		 and $$a_{N-1}^2 = (g_{N-1}+b_{N-1})/2 = (-1+b_{N-1})/2,$$
		 which give
		 \begin{equation}
		 a=-b_{N-1}^2,
		 \end{equation} and
		 \begin{equation}
		 b_{N-1}=2a_{N-1}^2+1,
		 \end{equation}
		 respectively.
		 Taking $n=N$ and $r=N-2$ in Theorem $2.8$ and using $(3.13)$, there exist $a_{N-2},b_{N-2}\in \mathbb{Q}(\sqrt{-b_{N-1}^2+1})$ such that \begin{equation}
		 g_{N-2}^2-b_{N-2}^2=-b_{N-2}^2=1+\sqrt{-b_{N-1}^2+1}
		 \end{equation} and 
		 \begin{equation}
		 2a_{N-2}^2 = b_{N-2}.
		 \end{equation} Let
		 \begin{equation}
		 b_{N-2}=x+y\sqrt{-b_{N-1}^2+1}
		 \end{equation}
		 for some $x,y\in \mathbb{Q}$. Using this in $(3.15)$, get  
		 \begin{equation}
		 x^2+y^2(-b_{N-1}^2+1)=-1
		 \end{equation} and 
		 \begin{equation}
		 -2xy=1.
		 \end{equation} Using $(3.19)$ in $(3.18)$ and simplifying, get
		 $$4x^4+4x^2-b_{N-1}^2+1=0,$$
		  which gives us 
		 $$b_{N-1}=\pm (2x^2+1).$$
		 Using $(3.14)$, $b_{N-1}$ is a positive integer, so \begin{equation}
		 b_{N-1}=2x^2+1
		 \end{equation} for some integer $x$. Again using $(3.14)$, get 
		 $$x = \pm a_{N-1}.$$
		  Since both cases give the same $a$ value (by $(3.13)$ and $(3.14)$), we can assume without loss of generality that 
		 \begin{equation}
		 x = a_{N-1}.
		 \end{equation}
		 Using $(3.19)$, this gives 
		 \begin{equation}
		 y = -\frac{1}{2a_{N-1}}.
		 \end{equation} Set 
		 $$a_{N-2} = \alpha + \beta \sqrt{-b_{N-1}^2+1}$$
		  for some $\alpha, \beta \in \mathbb{Q}$. Using this expression together with $(3.17)$, $(3.20)$ and $(3.21)$ in $(3.16)$, after simplifying, get 
		 \begin{equation}
		 2 (\alpha^2 + \beta^2 (-b_{N-1}^2 + 1)) = a_{N-1}
		 \end{equation} and 
		 \begin{equation}
		 4\alpha\beta = -\frac{1}{2a_{N-1}}.
		 \end{equation} Using $(3.14)$ and $(3.24)$ in $(3.23)$, and simplifying, get 
		 $$a_{N-1}^2 + 8\alpha^2 a_{N-1} -16\alpha^4 + 1 = 0.$$
		  Hence, 
		 $$a_{N-1} = -4\alpha^2 \pm \sqrt{32\alpha^4 - 1}.$$ 
		 Set $\alpha = \frac{p}{q}$ for $p,q \in \mathbb{Z}$ with $\gcd(p,q) = 1$. Using this, we get 
		 $$a_{N-1} = \frac{-4p^2 \pm \sqrt{32p^4 - q^4}}{q^2}.$$
		 $q$ must be even, as otherwise inside of the radical would be $-1$ (mod $4$), so set $q = 2r$. Simplifying, get 
		 $$a_{N-1} = \frac{-p^2 \pm \sqrt{2p^4 - r^4}}{r^2}.$$
		  Noting that $\gcd(p,q) = \gcd(p,r) =1$, a direct divisibility calculation implies that $r = \pm 1$. So, 
		 $$a_{N-1} = -p^2 \pm \sqrt{2p^4-1}.$$
		 Recall that $a_{N-1}$ is an integer, so
		 $$2p^4-1 = t^2$$
		 for some $t\in \mathbb{Z}$. It is known that this Diophantine equation has only $2$ positive integer solutions, namely $p = 1$ and $p = 13$ (see \cite{Ljunggren}). Using these solutions together with $(3.13)$ and $(3.14)$, we only get the solutions $a = -1^2, -9^2, -9801^2, -332929^2$. Note that $a = -1$ gives $f(x) = (x-1)^2$, which contradicts $f$ being irreducible. The remaining values were also forbidden at the beginning. Hence, we are done with this case.\\ 
		
		Next, we look at the case that $N$ is even. Note that $g_{N-1} = 0$. Since $f^N$ is irreducible and $f^{N+1}$ is reducible, taking $n=N$ and $r=N-1$ in Theorem $2.8$, there exists an integer $b_{N-1}$ such that
		$$g_{N-1}^2-b_{N-1}^2=-b_{N-1}^2=a+1,$$
		 hence 
		\begin{equation}
		a=-(b_{N-1}^2+1).
		\end{equation} Secondly, taking $n=N$ and $r=N-2$ in Theorem $2.8$ and using $(3.25)$, there exist $a_{N-2},b_{N-2}\in \mathbb{Q}(\sqrt{-b_{N-1}^2})=\mathbb{Q}(i)$ such that 
		$$g_{N-2}^2-b_{N-2}^2=1-b_{N-2}^2=1+\sqrt{-b_{N-1}^2}$$
		 and
		$$a_{N-2}^2=(g_{N-2}+b_{N-2})/2=(-1+b_{N-2})/2,$$
		which give 
		\begin{equation}
		-b_{N-2}^2=ib_{N-1}
		\end{equation} and 
		\begin{equation}
		b_{N-2}=2a_{N-2}^2+1,
		\end{equation}
		respectively. We can assume without loss of generality that $b_{N-1}$ is positive. Now, setting $b_{N-2}=z+ti$ for some $z,t\in \mathbb{Q}$, and using this in $(3.26)$, we get 
		\begin{equation}
		z=-t
		\end{equation}
		\begin{equation}
		b_{N-1}=2z^2.
		\end{equation} (Note that $(3.28)$ and $(3.29)$ also show that $z,t$ are integers, since $b_{N-1}$ is an integer.) Setting $a_{N-2}=u+iv$ for some $u,v\in \mathbb{Q}$, using this together with $(3.26)$ and $(3.29)$ in $(3.27)$, it follows that \begin{equation}
		2u^2-2v^2+1=-4uv=\pm z,
		\end{equation} which gives
		\begin{equation}
		2u^2+4vu-2v^2+1=0.
		\end{equation} Considering $(3.31)$ as a quadratic equation in $u$, get
		$$u=-v\pm\frac{\sqrt{8v^2-2}}{2}.$$
		 Using this together with $(3.30)$, and simplifying, get 
		\begin{equation}
		\pm z=4v^2\pm 2v\sqrt{8v^2-2}.
		\end{equation} Recall that $z$ is an integer. Set $v=\frac{p}{q}$ for some $p,q\in \mathbb{Z}$ such that $\gcd(p,q)=1$. Using this in $(3.32)$, and simplifying, get 
		\begin{equation}
		\pm z=\frac{p}{q^2}(4p\pm 2\sqrt{8p^2-2q^2}),
		\end{equation} hence 
		$$8p^2-2q^2=w^2$$
		 for some $w\in \mathbb{Z}$. Since $w$ must be even, put $w=2w_1$, and get 
		$$4p^2-q^2=2w_1^2.$$Since $q$ must be even, put $q=2q_1$ and get
		$$2(p^2-q_1^2)=w_1^2,$$
		 which gives
		$$2(p-q_1)(p+q_1)=w_1^2.$$
		Since $\gcd(p,q_1)=1$, we can assume without loss of generality that 
		\begin{equation}
		2(p-q_1)=\theta^2,
		\end{equation} and 
		\begin{equation}
		p+q_1=\gamma^2.
		\end{equation}Isolating $p$ and $q_1$ using $(3.34)$ and $(3.35)$, get
		$$p=\frac{\theta^2+2\gamma^2}{4},$$
		and
		$$q=2q_1=\frac{2\gamma^2-\theta^2}{2}.$$
		Note that $p$ and $q$ are both integers, so it follows that $\theta$ and $\gamma$ are both even. Put $\theta=2\theta_1$ and $\gamma=2\gamma_1$, hence get 
		$$p=\theta_1^2+2\gamma_1^2,$$
		 and 
		$$q=4\gamma_1^2-2\theta_1^2.$$
		Using these in $(3.33)$, get \begin{equation}
		\pm z = \frac{(\theta_1^2+2\gamma_1^2)(\theta_1^2+2\gamma_1^2\pm 4\theta_1\gamma_1)}{(2\gamma_1^2-\theta_1^2)^2}.
		\end{equation} Recall again that $z$ is an integer. Note that $\gcd(p,q)=1$, which implies that $\gcd(\theta_1,\gamma_1)=1$ as well. Since $z$ is an integer, it follows from an elementary divisibility calculation that 
		$$2\gamma_1^2-\theta_1^2 = \pm 2^b$$ for some non-negative integer $b$. However, if $b$ is positive, then $\theta_1$ will have to be even, which will make $p$ even, which will contradict with the fact that $p$ and q are relatively prime, since $q$ is also even. Hence, $b$ has to be $0$, and we get
		$$2\gamma_1^2-\theta_1^2=\pm 1.$$
		 Since two cases follow similarly, assume without loss of generality that 
		$$2\gamma_1^2-\theta_1^2=1.$$ Using this in $(3.36)$, get 
		$$\pm z=(4\gamma_1^2-1)(4\gamma_1^2-1\pm 4\theta_1\gamma_1)$$
		 for $\theta_1,\gamma_1\in \mathbb{Z}$ such that $2\gamma_1^2-\theta_1^2=1$. Setting $r=4\gamma_1^2-1$, and simplifying, this last equation gives
		$$\pm z=r(r\pm \sqrt{2(r^2-1)}).$$
        Combining this with $(3.25)$ and $(3.29)$, get
		$$a=-4(r^2\pm r\sqrt{2(r^2-1)})^4-1,$$
		 which, by Lemma $3.9$, implies that $f^3$ is reducible, contradicting our assumption. So, $f$ is stable if and only if $f^3$ is irreducible, as desired.
	\end{proof}
	\begin{proof}[Proof of Proposition \ref{thm:neat2}]
		It directly follows from Lemma $3.9$ and Lemma $3.10$.	
	\end{proof}
	\begin{remark}
		Except for the excluded three integer values of $a$, Theorem $3.8$ also answers the question posed at the end of (\cite{AyadMcdonald}, Remark $4$), namely the question of stability of an irreducible quadratic polynomial $f(x)=x^2+ax+b$ in the case $16|d_f$ and $\delta_f=-4$, as follows: It is clear that any polynomial of the form $(x+a)^2-a-1$ such that $a\equiv 3$ (mod $4$) gives $16|d_f$ and $\delta_f=-4$. To see the converse, suppose $f(x)=x^2+ax+b$ is such that $16|d_f$ and $\delta_f=-4$. We have $d_f = a^2-4b$ is divisible by $16$, set $a=2a_1$ for some $a_1\in \mathbb{Z}$. Using this and Definition $2.6$, we have
	    $$\delta_f=-d_f+4a_1=-4a_1^2+4b+4a_1=-4.$$
		 Simplifying, we get
		$$b=a_1^2-a_1-1,$$
		 hence $f(x)=x^2+2a_1x+a_1^2-a_1-1=(x+a_1)^2-a_1-1$, which finishes the proof.
	\end{remark}
    We found using MAGMA that the excluded three values in Theorem $3.8$ still make the first $10$ iterates of $f(x)$ irreducible, so based on this we make the following conjecture:
	\begin{conjecture}
		Let $S$ be the set given in Theorem $3.8$. Then for any $m\in S$, $f(x) = (x-m^2)^2+m^2-1 \in \mathbb{Z}[x]$ is stable.
	\end{conjecture}
We now restate Theorem $1.2$ and give its proof:
\begin{customthm}{1.2}
	Let $S = \{9,9801,332929\}$, and suppose that $g_a(x) \neq (x-m^2)^2 + m^2-1$ for any $m\in S$. Then we have the following:\\
	\item[(i)] All the iterates of $f_a$ are irreducible $\iff$ $f_a^2$ is irreducible.\\
	
	\item[(ii)] All the iterates of $g_a$ are irreducible $\iff$ $g_a^3$ is irreducible.\\
	
	\item[(iii)] All the iterates of $h_a$ are irreducible $\iff$ $h_a^2$ is irreducible.
\end{customthm}
\begin{proof}[Proof of Theorem \ref{thm:neat3}]
\item[(i)] 	This follows from Remark $3.6$.
\item[(ii)] This follows from Lemma $3.10$.
\item[(iii)] Recall from the proof of Proposition $3.7$ that since $\delta_f=-8$ in this case, we have $f$ is stable if and only if $f^2$ is irreducible (by (\cite{AyadMcdonald}, Remark $4$)), as desired.
\end{proof}
	We can finally prove Theorem $1.1$. We first restate the theorem using Notation $2.9$:
	\begin{customthm}{1.1}
    
	Let $f(x)\in X_a$ for some $a\in \mathbb{Z}$. Let $S = \{9, 9801, 332929\}$, and suppose that $f(x) \neq (x-m^2)^2 + m^2-1$ for any $m\in S$. Then $f^n(x)\in \mathbb{Z}[x]$ is irreducible for all $n$, and there exists $N\in \mathbb{N}$ such that $f^n(x)$ is reducible modulo every prime for all $n\geq N$ if and only if $f(x)$ has one of the following forms:\\
	
	\item[1)] $f(x)=(x-b^2)^2+b^2$ for some $b\in \mathbb{Z}$ such that $b\neq 2k^2$ for any $k\in \mathbb{Z}$.\\
	\item[2)] $f(x)=(x+b^2)^2-b^2-1$ for some $b\in \mathbb{Z}$.\\
	\item[3)] $f(x)=(x-b^2)^2+b^2-1$ for some $b\in \mathbb{Z}$ such that $b\neq 2k^2-1$ for any $k\in \mathbb{Z}$.\\
	\item[4)] $f(x)=(x-b^2-1)^2+b^2$ for some $b\in \mathbb{Z}$ such that $b\neq 2(r^2\pm r\sqrt{2(r^2-1)})^2$ for any integer solution $r$ of the Pell equation $2r^2-t^2=2$.\\
	\item[5)] $f(x)=(x+2-b^2)^2+b^2-4$ for some $b\in \mathbb{Z}$ such that $b\neq 2v^2-2$ for any $v\in \mathbb{Z}$.
	\end{customthm}
	\begin{proof}[Proof of Theorem \ref{thm:neat}]
		Recall that the union $\cup_{a\in \mathbb{Z}} X_a$ is equal to the set of all monic, PCF quadratic polynomials with integer coefficients. Using Lemma $3.1$ and Definition $3.2$, and also again noting that for any $a\in \mathbb{Z}$, all three polynomials in $X_a$ are reducible (mod $2$), it suffices to classify all $f(x)$ which are stable and special type. We will look at the cases $f_a, g_a$ and $h_a$ separately:\\
		
		\item[(i)] $f(x)=f_a(x)$. Recall that $f_a(x) = (x+a)^2-a$. For $f$ to be a special type polynomial, by Lemma $3.3$, it must be of the form $f(x)=(x-b^2)^2+b^2$ for some $b\in \mathbb{Z}$. Using this and Proposition $3.5$, $f$ is stable if and only if $b\neq 2k^2$ for any $k\in \mathbb{Z}$, which gives that $f$ is a stable and special type polynomial if and only if $f(x)=(x-b^2)^2+b^2$ for some $b\in \mathbb{Z}$ such that $b\neq 2k^2$ for any $k\in \mathbb{Z}$.\\
		\item[(ii)] $f(x) = g_a(x)$. Recall that $g_a(x)=(x+a)^2-a-1$. For $f$ to be a special type polynomial, by Lemma $3.3$, it must be of the form $f(x)=(x+b^2)^2-b^2-1$, $(x-b^2)^2+b^2-1$ or $(x-b^2-1)^2+b^2$ for some $b\in \mathbb{Z}$. Using this and Theorem $3.8$, $f$ is a stable and special type polynomial if and only if $f(x)=(x+b^2)^2-b^2-1$ for some $b\in \mathbb{Z}$ or $f(x)=(x-b^2)+b^2-1$ such that $b\neq 2k^2-1$ for any $k\in \mathbb{Z}$ or $f(x)=(x-b^2-1)^2+b^2$ such that $b\neq 2(r^2\pm r\sqrt{2(r^2-1)})^2$ for any integer solution $r$ of the Pell equation $2r^2-t^2=2$.\\
		\item[(iii)] $f(x)=(x+a)^2-a-2$. For $f$ to be a special type polynomial, by Lemma $3.3$, it must be of the form $f(x)=(x+2-b^2)^2+b^2-4$ for some $b\in \mathbb{Z}$. Using this and Proposition $3.7$, $f$ is a stable and special type polynomial if and only if $f(x)=(x+2-b^2)^2+b^2-4$ such that $b\neq 2v^2-2$ for any $v\in \mathbb{Z}$, which completes the proof.
	\end{proof}
	\section{A rigidity conjecture for stability}
	In this last section, we first give a generalization of Proposition $3.7$, and then make a conjecture based on the results we have found so far.\\
	
	Let $K$ be any field of characteristic $\neq 2$, and $f(x)\in K[x]$ a quadratic polynomial. Recall that we denote by $o_f$ and $t_f$ the critical orbit size and tail size of $f$, respectively.
	
	\begin{theorem}
		Let $K$ be any field of characteristic $\neq 2$, and $f(x)\in K[x]$ be a monic, PCF quadratic polynomial. Suppose that $t_f=1$. Then $f$ is stable if and only if $f^{o_f}$ is irreducible. 
	\end{theorem}
	
	\begin{proof}
		One direction is obvious, so we can assume that $f^{o_f}$ is irreducible. We will show that $f$ is stable.\\
		
		We can assume $o_f\geq 2$, as otherwise $t_f$ cannot be $1$. Let $f(x)=(x-c)^2+c+m$ for some $c,m\in K$, where $c$ is the critical point of $f$. So, the post-critical orbit of $f$ becomes 
		
		$$O_f = \{c+m, c+m^2+m, c+(m^2+m)^2+m,\dots\}.$$
		 Setting $f_m(x)=x^2+m$ gives 
		$$O_f = \{f_m(0)+c, f_m^2(0)+c, f_m^3(0)+c,\dots\}.$$
		 Since $t_f = 1$, we have 
		$$f_m^{o_f}(0) = f_m^2(0).$$
		$$\iff (f_m^{o_f-1}(0))^2+m = m^2+m.$$
		$$\iff f_m^{o_f-1}(0) = \pm m.$$
		 The fact that $t_f = 1$ directly implies $f_m^{o_f-1}(0)=-m$. Hence, it follows that the sequence $-m, m^2+m, (m^2+m)^2+m, \dots$ is periodic with period $o_f-1$. \\
		
		With the notation in Definition $2.6$, we have $\delta_f = 4m$, hence $\delta_0 = m$. So, again using the notation in Definition $2.6$, the corresponding sequence $\{g_i\}_{i\geq 0}$ is given by $g_0 = -m$, $g_i=(g_{i-1})^2+m$ for $i\geq 1$. Thus, the sequence $\{g_i\}_{i\geq 0}$ is periodic by the first part. Suppose for the sake of contradiction that there exists $N>o_f$ such that $f^N$ is irreducible, but  $f^{N+1}$ is reducible. By Theorem $2.8$, for every $r$, $0\leq r \leq N-1$, there exist elements $a_r, b_r \in K_{N-r-1}$ such that $g_r^2-b_r^2 = d_{n-r-1}$ and $a_r^2 = \frac{g_r + b_r}{2}$. We now let $N'$ be the integer such that $1\leq N'\leq o_f-1$ and $N\equiv N'$ (mod $o_f-1$). Since the sequence $\{g_i\}_{i\geq 0}$ is periodic, applying the converse part of Theorem $2.8$ to $f^{N'}$, we get that $f^{N'+1}$ is reducible, which contradicts the initial assumption, since $N'+1\leq o_f-1+1 = o_f$. Thus, $f$ is stable if and only if $f^{o_f}$ is irreducible, as desired. 
	\end{proof}
	
	\begin{remark}
		In Proposition $3.7$, it is proven that $f$ is stable if and only if $f^2$ is irreducible for the polynomials of the form $f(x)=(x+a)^2-a-2$. Note that $t_f=1$ and $o_f = 2$ in that particular case. Hence, it follows that Proposition $3.7$ is a special case of Theorem $4.1$.
	\end{remark}
	\begin{remark}
	It is natural to ask whether Theorem $1.2$ generalizes to post-critically finite polynomials with rational coefficients, since those are also of the form $(x+a)^2-a$, $(x+a)^2-a-1$ or $(x+a)^2-a-2$, $a\in \mathbb{Q}$. It follows from Theorem $4.1$ that this generalization is true for the family $(x+a)^2-a-2$. It is also clear from the proof of Proposition $3.5$ that same is true for the family $(x+a)^2-a$. However, our methods do not work for the polynomials of the form $(x+a)^2-a-1$, $a\in \mathbb{Q}$, because the proof of Lemma $3.10$ relies heavily on the fact that $a$ is an integer. In particular, if we do not assume that $a$ is integer, in the first part of the proof of Lemma $3.10$, we come across the elliptic curve $y^2=2x^4-1$, which has infinitely many rational points, none of which can be dealt with our methods. We get infinitely many such problematic values of $a$ from the second part of the proof of Lemma $3.10$ as well. So, this last case remains open. 
	\end{remark}
	Based on the results we have so far, we make the following conjecture:
	
	\begin{conjecture}
		Let $K$ be any field of characteristic $\neq 2$, and $f(x)\in K[x]$ be a monic, PCF quadratic polynomial. Then we have the following:\\
		
		\item[i)]  If $t_f = 0$, $f$ is stable if and only if $f^{o_f+1}$ is irreducible.
		\item[ii)] If $t_f > 0$, $f$ is stable if and only if $f^{o_f}$ is irreducible. 
		
	\end{conjecture}
	
	\begin{remark}
		The exact analogue of Conjecture 4.4 already holds over finite fields of odd characteristic, and it was indeed part of our motivation to state this conjecture. Namely, let $\mathbb{F}_q$ be the finite field of odd size $q$, and $f(x)\in \mathbb{F}_q [x]$ be a monic quadratic polynomial with the critical point $c$. Then for any $n\geq 2$, using (\cite{JonesBoston}, Lemma $2.5$), $f$ is stable if and only if the elements of the sequence $$\{-f(c), f^2(c), \dots, f^{o_f}(c), f^{o_f+1}(c)\}\subseteq \mathbb{F}_q$$ are all non-squares in $\mathbb{F}_q$. If $t_f=0$, we have $f^{o_f+1}(c) \neq -f(c)$ and $f^{o_f+1}(c) \neq f^{i}(c)$ for all $2\leq i\leq o_f$. Then (\cite{JonesBoston}, Lemma $2.5$) directly implies that $f$ is stable if and only if $f^{o_f+1}$ is irreducible. If $t_f>0$, then $f^{o_f+1}(c) = -f(c)$ or $f^{o_f+1}(c) = f^i(c)$ for some $2\leq i\leq o_f$. Then (\cite{JonesBoston}, Lemma $2.5$) again implies that $f$ is stable if and only if $f^{o_k}$ is irreducible.
	\end{remark}
     \subsection*{Acknowledgments}
     The author would like to thank Nigel Boston and Rafe Jones for their very helpful comments on this work. The author also thanks the anonymous referee for their careful reading of the manuscript, and many helpful suggestions.
	
		%% Use the widest label as parameter above.
		%% Reference items can be numbered or have labels of your choice, as below.
		%% Arrange the items in the alphabetical order of names (and not in the order of labels).
		
		%% In IMPAN journals, only the title is italicized; boldface is not used.
		%% Do NOT give the issue number unless the issues are paginated separately, as in Uspekhi below.
		
		%% To ease editing, add:
		
		\baselineskip=17pt

		%%%%%%%%%%%%%

\end{document}